\documentclass[11pt,oneside,reqno]{amsart}
\usepackage[utf8]{inputenc}
\pdfoutput=1
\usepackage{parskip}
\usepackage{mathrsfs}
\usepackage[margin=1in]{geometry}
\usepackage{enumerate}
\usepackage{amsthm}
\usepackage{amsmath}
\usepackage{amsfonts}
\usepackage{amssymb}
\usepackage{graphicx}
\usepackage{color}
\usepackage{graphics}
\usepackage{eepic}
\usepackage{nicefrac}
\usepackage{bigints}
\usepackage{bbm}
\usepackage{array}
\usepackage{caption}

\makeatletter
\@addtoreset{equation}{chapter}
\makeatother
\usepackage{tikzsymbols}
\usepackage{booktabs}
\usepackage{hyperref}
\hypersetup{
 colorlinks   = true,
 urlcolor     = blue,
 linkcolor    = blue,
 citecolor   = red,
 bookmarksopen=true
}

\newcommand{\ignore}[1]{}

\usepackage[colorinlistoftodos]{todonotes}

\newcommand{\be}{\begin{equation}}
\newcommand{\ee}{\end{equation}}

\newcommand{\C}{{\mathbb{C}}}

\newcommand{\N}{{\mathbb{N}}}

\newcommand{\D}{{\mathbb{D}}}

\usepackage{mathtools} 

\newcommand{\diam}{\mbox{diam}}

\newtheorem{thm}{Theorem}[section]

\newtheorem{prop}[thm]{Proposition}

\newtheorem{Lemma}[thm]{Lemma}

\newtheorem*{claim*}{Claim}

\newtheorem{remark}[thm]{Remark}

\theoremstyle{definition}
\newtheorem{defn}[thm]{Definition}
\newtheorem{example}[thm]{Example}

\theoremstyle{remark}

 \renewcommand\epsilon{\varepsilon}

\usepackage{palatino}
 \usepackage{color, xcolor, mathrsfs, graphicx, hyperref, framed}

\newcommand{\mb}{\mathbb}

\usepackage[italian,english]{babel}
\usepackage[autostyle]{csquotes}  
\author{Subhajit Ghosh}
\address{Department of Mathematics, Bar-Ilan University, Ramat Gan, 5290002, Israel
 }
 \email{ghoshsu1@biu.ac.il}

\author{Koushik Ramachandran}
\address{Tata Institute of Fundamental Research, Centre for Applicable Mathematics, Bangalore, India-560065}
\email{koushik@tifrbng.res.in}

\title{A note on the Erd\H{o}s minimal area problem}

\begin{document}
\begin{abstract}
We answer a question of Erd\H{o}s, Herzog, and Piranian \cite[Problem $4$, page $135$]{EHP} on the minimal area of polynomial lemniscates when all the zeros of the polynomial are constrained to lie on a compact set $K$ whose logarithmic capacity is greater than $1$. We show that the minimal area decays exponentially fast and prove that the Fekete polynomials are asymptotic minimizers for the area problem.
\end{abstract}
\maketitle

\section{Introduction}
Let $p$ be a complex polynomial. The unit lemniscate of $p$ is defined to be the set $\Lambda_p = \{z\in\C: |p(z)| < 1\}$. We note that $\Lambda_p$ is a bounded open set and hence its area is non-zero. Let $K\subset\C$ be a compact set.  For $n\in\N,$ let $\mathscr{P}_n(K)$ denote the space of \emph{monic} polynomials of degree $n$, all of whose zeros lie in the compact $K,$ i.e.,
\[\mathscr{P}_n(K) = \bigg\{p(z): p(z) = \prod_{j=1}^{n}(z - z_j), \hspace{0.05in}\mbox{where}\hspace{0.05in} z_j\in K \hspace{0.05in}\mbox{for}\hspace{0.05in} 1\leq j\leq n\bigg\}.\]
Let $m$ denote Lebesgue measure in the complex plane. Following in the vein of \cite{EHP} we consider the minimal lemniscate area for the constrained problem and define
\begin{align}\label{def A}
    A_n(K) := \inf_{p\in \mathscr{P}_n(K)} m(\Lambda_p), \qquad \vartheta(K):= \inf_{n} A_n(K).
\end{align}

In \cite[Problem $4$, page $135$]{EHP} Erd\H{o}s, Herzog, and Piranian posed questions regarding the relationship between $\vartheta(K)$ and the logarithmic capacity (equivalently, the transfinite diameter), which we will denote by $\operatorname{cap}(K)$. Throughout the paper, we shall use tools from potential theory in the complex plane. For the benefit of the reader, we briefly digress to recall the definition of logarithmic potential of a measure and the capacity of a compact set. Given a positive measure $\mu$ with compact support in $K$, its logarithmic potential is defined by $U_{\mu}(z) = \int_{K}\log|z-w| d\mu(w), \hspace{0.05in} z\in\C$. Of relevance to us is the fact that if $Z$ is a finite multiset of cardinality $n$, and $\mu = \sum_{z\in Z}\frac{1}{n}\delta_{z}$, then $U_{\mu}(z)$ is simply $\frac{1}{n}\log|p(z)|$, where $p$ is the monic polynomial of degree $n$ with zero set $Z$. The capacity of $K$ is defined by

\begin{align}\label{log cap}
    \log \operatorname{cap}(K) := \sup_{\mu\in\mathcal{P}(K)}\int_{K}U_{\mu}(z)d\mu(z),
\end{align}
where $\mathcal{P}(K)$ denotes the collection of all Borel probability measures on $K.$
One can show that for reasonable compact sets, there is a unique measure (called the \emph{equilibrium measure}) $\nu_K\in\mathcal{P}(K),$ which realizes the supremum in the above definition. For more on capacities and their applications to analysis, we refer the reader to \cite{Ransford}. 
 
We are now ready to go back to the question posed by Erd\H{o}s, Herzog, and Piranian. They asked

\begin{enumerate}[(I)]
    \item Is $\vartheta(K)$ uniquely determined by the logarithmic capacity $\operatorname{cap}(K)$?
    \item Is it true that $\vartheta(K) = 0$ whenever $\operatorname{cap}(K) \ge 1$?
\end{enumerate}

This problem is also listed as Problem $1040$ in the recently curated Erd\H{o}s problems  website \cite{ErdosProblemsBloom}. In this paper, we answer both questions (the second one partially). First we prove that (I) does not hold in general. In a recent preprint \cite{Aletheia}, with the help of an AI tool, it was shown that there exist two countably infinite compact sets (hence having capacity zero), $K_1$ and $K_2$,  with $\vartheta(K_1)\neq \vartheta(K_2)$. Here, we show that given $t\in (0, 1)$ it is quite easy to construct compacts $K_1, K_2$ with
\[
\operatorname{cap}(K_1) = \operatorname{cap}(K_2)=t, \quad \text{yet} \quad \vartheta(K_1) \neq \vartheta(K_2),
\]
These examples illustrate that not only the capacity but other geometric/metric features of $K$ also play a role in determining $\vartheta(K)$.

\vspace{0.1in}

\noindent It was proved in \cite[Corollary $1.6$]{KLR1} that (II) holds provided $\operatorname{cap}(K) > 1$ and, in addition there exist constants $C, \eta > 0$ such that the equilibrium measure $\nu$ of $K$ satisfies an estimate of the form $\nu(B(z, r))\leq C r^{\eta}$, for all balls $B(z, r)$. Our main result, Theorem \ref{main} below, shows that (II) is true when $\operatorname{cap}(K)>1$ without any other restrictions on the regularity of $K$. We in fact prove that in this case, the well known Fekete polynomials of degree $n$ have lemniscate areas going to zero exponentially fast in $n$ as $n$ tends to infinity, thereby obtaining an upper bound on $A_n(K)$. We follow this result by proving a lower bound on $A_n(K)$ that asymptotically matches the upper bound. Combining the two results gives 

\begin{equation}\label{universallaw}
\lim_{n\to\infty}\frac{\log A_n(K)}{n}= -2\log \operatorname{cap}(K).    
\end{equation}

We stress here two important points regarding the relation \eqref{universallaw}: the first is that this result is true for \emph{all} compacts with capacity bigger than one, no matter how non-smooth they are. Secondly, the limit depends only on the capacity of the set, and no other features of $K$. This is a good place to mention that as a consequence of the proof of \eqref{universallaw}, it will follow that the sequence of Fekete polynomials are asymptotic area minimizers, i.e., the lemniscate areas of the Fekete polynomials satisfy the same limiting relation as in \eqref{universallaw} .

\vspace{0.1in}

It is worthwhile to remark here that when $\operatorname{cap}(K) = 1$, the Fekete polynomials do not necessarily have small areas for their lemniscates. To see this, just consider $K = \overline{\D}$, where $p_n^* (z) = z^n -1$ are the unique (upto rotation) Fekete polynomials. Now it is known that for all $n\in \N$, the corresponding lemniscates $\Lambda_{{p_n^*}}$ have area bounded away from $0$. In their paper \cite{EHP}, Erd\H{o}s, Herzog and Piranian proved that $\vartheta(\overline{\D}) = 0$. However even for $K = \overline{\D}$, there are no explicit examples of a sequence $p_n\in\mathscr{P}_n(K)$ with $m(\Lambda_{p_n})\to 0$. This already hints that the case when $K$ has capacity $1$ is a bit more delicate to handle. However, see the recent work \cite[Theorem $6$]{KLR2} where Krishnapur, Lundberg, and Ramachandran have shown using an approximation argument that when $K$ is the closure of a bounded $C^2$ smooth domain with $\operatorname{cap}(K) = 1$, one has $\vartheta(K) = 0$. The result for general capacity one sets remains open. 

\vspace{0.1in}

\noindent In the last section of our article we show that $A_n([-2, 2])$ goes to zero like $\frac{1}{n}$ and compare it with the known decay of $A_n(\overline{\D})$ which is like $\frac{1}{\log n}$ to illustrate that for compacts $K$ with $\operatorname{cap}(K) = 1,$ the decay rate of $A_n(K)$ is not universal. We note that this is quite unlike the scenario for sets with $\operatorname{cap}(K) > 1$, where the relation \eqref{universallaw} proves universality of the decay rate.

\section{\texorpdfstring{Capacity does not determine $\vartheta(K)$}{Capacity does not determine mu(K)}}
In \cite{ErdosNetanyahu}, Erd\H{o}s and Netanyahu proved that if the compact set $K$ is \emph{connected} and $\operatorname{cap}(K)\in (0, 1)$, then $\Lambda_{p}$ contains a disk of positive radius $r$, where $r$ depends only on $\operatorname{cap}(K)$. In particular, let us take $K = \overline{B(0,t)}$. It is well known that $\operatorname{cap}(K) = t$. It follows from \cite{ErdosNetanyahu} that
\[
\vartheta(K)\ge \operatorname{Area}(B(0,r(t))) = \pi r(t)^2.
\] 
We now give an example of a disconnected compact set $K$ with capacity $t\in(0, 1)$ for which $\vartheta(K)$ can be made smaller than any prescribed positive number.

\begin{example}\label{ex1}
Let $t\in (0, 1)$ and $j\geq 2$ be fixed. Let $\varepsilon = \varepsilon(t, j) > 0$ be a constant to be chosen later. Consider the symmetric intervals $B_1=[j, j+\varepsilon]$ and $ B_2=[-j -\varepsilon,-j]$. Set $B:=B_1\cup B_2.$
By \cite[Corollary 5.2.6]{Ransford},
\[
\operatorname{cap}(B)=\frac12\sqrt{(2j+\varepsilon)\varepsilon}.
\]
Now we choose $\varepsilon >0$ so that $\operatorname{cap}(B)=t$. Next, consider the sequence of polynomials
\begin{align}\label{poly}
    p_n(z)=
\begin{cases}
(z-j)^l(z+j)^l, & \text{if } n=2l,\\[1mm]
(z-j)^{l+1}(z+j)^l, & \text{if } n=2l+1.
\end{cases}
\end{align}
We claim that 
\begin{align}\label{smallarea}
\Lambda_{p_n}\subset B\!\left(-j,\frac{1}{\sqrt{j}}\right)\cup B\!\left(j,\frac{1}{\sqrt{j}}\right).
\
\end{align}
\begin{figure}
    \centering
    \includegraphics[width=\textwidth]{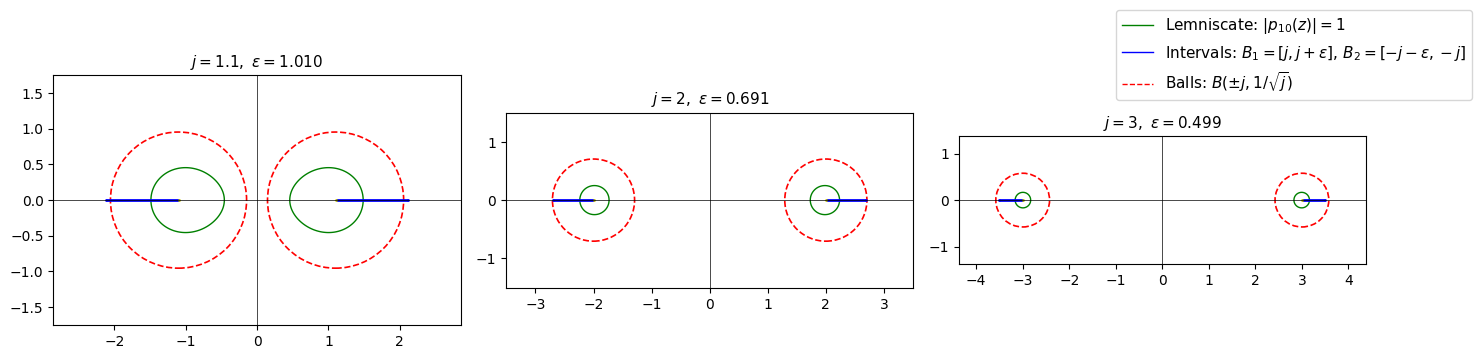}
     \caption{The compact set shown in the figure is $[j,j+\varepsilon]\, \cup \, [-j-\varepsilon,-j],$ where $j$ and $\varepsilon$ are chosen so that its logarithmic capacity remains fixed at $0.9$. The polynomial is as defined in \eqref{poly}. The figures illustrate that, as $j$ increases, the lemniscate shrinks toward the two zeros of the polynomial, and the total area enclosed by the level set decreases.}
    \label{fig:counterexample}
\end{figure}
\noindent Assume for a moment that \eqref{smallarea} holds. Then for all $n$ we will have
$A_n(K)\le \frac{2\pi}{j}$, which in turn implies that $\vartheta(K)\le \frac{2\pi}{j}.$ Since $j$ can be taken arbitrarily large, this will yield our example. See Figure \ref{fig:counterexample}. For the proof of the claim, suppose $z \notin  B\!\left(-j,\frac{1}{\sqrt{j}}\right)\cup B\!\left(j,\frac{1}{\sqrt{j}}\right)$. Then $w=\min\{|z-j|,|z+j|\}\geq \frac{1}{\sqrt{j}}$. Now we handle two cases. If $w>\sqrt{j}$, \[ |p_n(z)|> \frac{1}{\sqrt{j}}|z-j|^l|z+j|^l > \frac{1}{\sqrt{j}}j^l>1.\] 
Hence $z\notin \Lambda_{p_n}$. On the other hand if $w\leq\sqrt{j}$, say $|z-j|\leq\sqrt{j}$, then by the triangle inequality $|z+j|\geq 2j - w\geq j.$ Therefore
 \[ |p_n(z)|> \frac{1}{\sqrt{j}}|z-j|^l|z+j|^l > \frac{1}{\sqrt{j}} |2j-w|^lw^l> \frac{ j^l}{(\sqrt{j})^{l+1}}> 1.\] 
 Hence once again we have $z\notin\Lambda_{p_n}$ and the proof of \eqref{smallarea} is complete.

\end{example}

\section{Main theorem}
\begin{thm}\label{main}
Let $t > 1$ be fixed. Let $K\subset\C$ be a compact set with $\operatorname{cap}(K)= t$. Then,
\[\limsup_{n\to\infty}\frac{\log(A_n(K))}{n}\leq - 2\log t.\]
In particular, $\vartheta(K) = 0$.
\end{thm}

\begin{proof}
The proof is adapted from the argument in \cite[Theorem $2.2$ (b)]{GhoshRamachandran}. We begin by recalling that a Fekete $n$-tuple of $K$ is any $n$-tuple $(z_1, z_2, ..., z_n)$ which realizes the supremum

\[\sup_{j, k\leq n : j < k}\left\{\prod |w_j - w_k|^{\frac{2}{n(n-1)}}: w_1, w_2, ... w_n\in K\right\}.\]

We refer to \cite[Section $5.5$]{Ransford} for more on Fekete points and their connection to potential theory. For each $n \in \mb{N}$, let $\{z_{k,n}\}_{k=1}^{n}$ denote a Fekete $n$-tuple for the compact set $K$. Let $f_n$ denote the corresponding Fekete polynomial of degree $n$ defined by
    \begin{align*}
        f_n(z) = \prod_{k=1}^n (z-z_{k,n}), \quad \quad n \in \mb{N}.
    \end{align*}
    We prove that $\{|f_n| < 1\}$ has small area. As a first step, we show that for $1\leq j\leq n,$ the size of the derivative $\big|f_n'(z_{j,n})\big|$ is exponentially large using properties of the Fekete points. Let us consider a zero, say $z_{1, n}$. Define
    \begin{align*}
        q_1(z) := \prod_{j=2}^{n} \frac{(z-z_{j,n})}{(z_{1,n}-z_{j,n})}.
    \end{align*}
    From the definition of Fekete points we infer that for each $z\in K$, the product of the pairwise distances between the $n$ points $\{z, z_{2,n}, z_{3, n}, \ldots, z_{n, n}\}$ is not larger than the corresponding product for the Fekete $n$-tuple. This implies that 
    \[\sup_{z\in K}\prod_{j=2}^{n}|z - z_{j, n}|\leq \prod_{j= 2}^{n} \left| z_{1,n}-z_{j,n}\right|.\] Now the denominator in the expression for $q_1$ is simply $f_n'(z_{1,n})$. Using this along with the previous estimate, we obtain
    \begin{align}\label{fekete derivative estimate}
        |f'_n(z_{1,n})| =\prod_{j= 2}^{n} \left| z_{1,n}-z_{j,n}\right| \geq\sup_{z\in K}\prod_{j=2}^{n}|z - z_{j, n}| \geq \operatorname{cap}(K)^{n-1}\geq t^{n-1}, 
    \end{align}
    where the second inequality in \eqref{fekete derivative estimate} follows from the fact that the sup norm of a monic polynomial of degree $d$ on a compact $K$ is at least $\operatorname{cap}(K)^{d}$, see \cite[Theorem $5.5.4$]{Ransford}. This establishes the largeness of the derivative at $z_{1,n}$.

\vspace{0.1in}

\noindent  Next, denote by $G_1$ the connected component of $\Lambda_{f_n}$ that contains $z_{1, n}$. Let $d_1$ be the diameter of $G_1$. Then using $\operatorname{cap}(G_1)\geq \frac{d_1}{4}$ (see \cite[Theorem $5.3.2$]{Ransford}), along with Bernstein's inequality \cite{PommDeriv}, we have that for some absolute constant $A>0$,
\begin{align}\label{smalldiam}
t^{n-1}\leq  |f_n'(z_{1, n})|\leq \frac{An^2||f_n||_{\overline{G_1}}}{d_1} = \frac{An^2}{d_1}.
\end{align}
We note in passing that the Bernstein inequality relating the size of the derivative to the sup norm of the polynomial holds for any connected compact set, irrespective of the regularity of the set.  The estimate \eqref{smalldiam} shows that $d_1\leq\exp\left[-n\log t (1 + o(1))\right]$. Since $G_1\subset B(z_{1,n}, d_1)$, it follows that $m(G_1)\leq \exp\left[-2n\log t (1 + o(1))\right] $. A similar area estimate obviously holds for every component of $\Lambda_{f_n}$. Since $\Lambda_{f_n}$ can have at most $n$ connected components, the preceding estimates along with the subadditivity of area gives $m(\Lambda_{f_n})\leq \exp\left[-2n\log t (1 + o(1))\right]$. This in turn implies that the same upper bound holds for $A_n(K)$ as it is the minimal area. Taking logarithms, dividing by $n$ and then taking limsup of the resulting sequence as $n\to\infty$ concludes the proof.
\end{proof}
\begin{remark}\label{remark1}
    The proof shows that the result above continues to hold if instead of minimizing the area of the level one lemniscate $\{|p| < 1\}$, we consider the corresponding minimum area (for $p\in\mathscr{P}_n(K)$) of lemniscates of the form $\{|p| < l_n\}$, where $l_n$ can grow as fast as $\exp(n^{\gamma})$ for $\gamma < 1$.
\end{remark}


We will next prove that the bound for $A_n(K)$ from Theorem \ref{main} is asymptotically sharp. As mentioned earlier, as a consequence we can show Fekete polynomials are \emph{asymptotic} minimizers for the area problem.  See Proposition \ref{nonminimizer} below for more on this.



\begin{thm}\label{asymptotic min}
    Suppose the assumptions of Theorem \ref{main} hold. Let $\{f_n\}_{n \in \N}$ be a sequence of Fekete polynomials for the compact set $K$. Then
    \begin{align} \label{log cap limit}
        \lim_{n \to \infty} \frac{1}{n}\log A_n(K) = \lim_{n \to \infty} \frac{1}{n}\log m(\Lambda_{f_n})= -2\log t.
    \end{align}
\end{thm} 

The proof of Theorem \ref{asymptotic min} relies on an important fact concerning the limiting distribution of zeros of any sequence of area minimizers (for similar results in the setting of capacity $1$ sets, see \cite[Theorem~$7$]{KLR2}). In what follows below, we write $\mu_n \xrightarrow[]{w^\ast} \mu$ to denote weak$^*$ convergence of compactly supported measures.

\begin{prop}\label{equidistribution}
Let $t > 1$ be fixed. Let $K \subset \mathbb{C}$ be a compact set with $\operatorname{cap}(K)= t.$ For each $n \in \N$, let $q_n$ be an area minimizing monic polynomial of degree $n$, i.e.,
\[
    m(\Lambda_{q_n}) = A_n(K).
\]

Let $\mu_n := \frac{1}{n}\underset{{z\,:\,q_n(z)=0}}{\sum} \delta_z$
be the empirical measure of the zeros of $q_n$. Then
\begin{align}\label{convergence to equilibrium measure}
    \mu_n \xrightarrow[]{w^\ast} \nu_K,
\end{align}
where $\nu_K$ denotes the equilibrium measure of $K$.
\end{prop}

The proof of  Proposition \ref{equidistribution} in turn relies on the following lemma.
\begin{Lemma}\cite[Lemma~3.5]{GhoshRamachandran}\label{Main lemma from JMAA}
    Let $\mu_n \in \mathcal{P}(K)$ be a sequence such that $\mu_n \overset{\omega^{*}}{\longrightarrow } \mu $. Then for any compact set $L$ we have,
    \begin{align*}
       \underset{n \to \infty}{\overline{\lim}}  \sup_{L} U_{\mu_n} \leq \sup_{L} U_{\mu}.
    \end{align*}
   
\end{Lemma}
\begin{proof}[Proof of Proposition \ref{equidistribution}]
    Assume that \eqref{convergence to equilibrium measure} does not hold. Then there exists a  subsequence $\{n_k\}$ (which henceforth will be assumed to be $\{n\}$ for notational simplicity) such that $\mu_{n} \overset{\omega^{*}}{\longrightarrow } \mu $ for some $\mu \neq \nu_K \in \mathcal{P}(K)$ (c.f. \cite[Theorem A.$4.2$]{Ransford}). Now from the definition \eqref{log cap} and the uniqueness of equilibrium measure, we have
    \begin{align*}
    \int_{K}U_{\mu}(w)d\mu(w)< \log t.
    \end{align*}
    This implies that there exists a point \(w_0 \in \operatorname{supp}(\mu)\) such that $U_{\mu}(w_0) < \log t.$
By the upper semicontinuity of \(U_{\mu}\), it follows that \(U_{\mu}<\log t\) in a neighborhood of \(w_0\). Hence, we can find a closed disk \(\overline{B}_r\) of radius $r$ centered at \(w_0\) such that
\[
U_{\mu}< \log t-\epsilon_0 \quad \text{on } \overline{B}_r,
\]
for some $\epsilon_0>0$ small enough. Applying Lemma~\ref{Main lemma from JMAA} with $L=\overline{B}_r$, we get 
\begin{align*}
       \underset{n \to \infty}{\overline{\lim}}  \sup_{\overline{B}_r} U_{\mu_{n}} \leq \sup_{\overline{B}_r} U_{\mu}< \log t-\epsilon_0 .
    \end{align*}
Since $U_{\mu_n}(z)= \frac{1}{n}\log |q_n(z)|$, the above estimate implies that we can find a $N_0 \in \N$ such that for all $n\geq N_0$
\begin{align*}
\sup_{\overline{B}_r}|q_{n}|< \exp\left({n(\log t-\frac{\epsilon_0}{2})}\right).
\end{align*}
Set $\lambda:=\exp(\log t-\frac{\epsilon_0}{2})$. Note that $\lambda < t$.  Let $\Omega$ be the connected component of $\{|q_n|<\lambda^n\}$ containing $B_r$. Next, we estimate the derivative of $q_n$ in  $\Omega$ using Bernstein's inequality \cite{PommDeriv} to obtain
\begin{align}\label{derivative estimate}
\sup_{\Omega}|q'_n| \leq C\frac{n^2}{\diam(\Omega)}\sup_{\Omega}|q_n|<C\frac{n^2}{r} \lambda^n
\end{align}
Let $\Omega_1$ be any connected component of $\Lambda_{q_n}$ that is contained in $\Omega$. Since $\Omega_1$ contains at least one zero of $q_n$, the area formula gives
\begin{align*}
    \int_{\Omega_1} |q'_n(z)|^2 dA(z) \geq \pi.
\end{align*}
We now plug the derivative bound from \eqref{derivative estimate} to estimate the integrand and obtain $$C^2n^4\frac{\lambda^{2n}}{r^2}m(\Omega_1)\geq \pi.$$ In other words, $m(\Lambda_{q_n}) \geq m(\Omega_1) \geq \frac{Cr^2}{n^4 \lambda^{2n}}.$
Since $\lambda < t$, this contradicts the upper bound obtained for $A_n(K)$ in Theorem~\ref{main}. Hence the desired assertion holds.
\end{proof}

We now go on to prove the Theorem.

\begin{proof}[Proof of Theorem \ref{asymptotic min}]
    The upper bound of the limit in \eqref{log cap limit} follows from Theorem~\ref{main}. So it will be enough to prove that $\liminf_{n\to\infty}\frac{\log A_n(K)}{n}\geq - 2\log t$. In fact we will prove a slightly more general result that we state as a Claim.
    
    \textbf{Claim:} Let $\{P_n\}$ be a sequence of polynomials in $\mathscr{P}_n(K)$ such that the the empirical measure of the zeros of $P_n$ converge in the weak$^*$ sense to the equilibrium measure $\nu_K$. Then
    \begin{align}\label{lower limit}
        \liminf_{n \to \infty} \frac{1}{n}\log m(\Lambda_{P_n}) \geq -2\log t
    \end{align}

    Of course, by Proposition \ref{equidistribution}, area minimizers have the property stated in the Claim and this will give Theorem \ref{asymptotic min}.

\noindent To convey the main idea to the reader, we will first sketch the proof when $K=\overline{\mathcal{D}}$, where $\mathcal{D}$ is a bounded domain with $C^2$ boundary. Let $P_n(z) = (z-z_{1,n})p_n(z)$. Then of course we have the empirical measure $\mu_{p_n}\xrightarrow[]{w^\ast}\nu_K$. Since $\mathcal{D}$ is assumed smooth, Frostman's Theorem \cite[Theorem $3.3.4$]{Ransford} guarantees that $U_{\nu_K}(z) = \log t$ for all $z\in K$. Applying Lemma \ref{Main lemma from JMAA} with $L = K$, we obtain $\limsup_{n\to\infty}\sup_{K}U_{\mu_{p_n}}\leq\sup_{K} U_{\nu_K} = \log t$. This means that for all large $n$, we have $||p_n||_{\overline{\mathcal{D}}}\leq (t + \epsilon_n)^n$, for some $\epsilon_n = o(1)$. Now we may assume that $z_{1,n}\in\partial\mathcal{D}$ (clearly that is the worst case). Then for $z\in B\left(z_{1,n},\frac{1}{2(t+\epsilon_n)^n}\right)\cap\mathcal{D}$, we have $|P_n(z)| = |z- z_{1,n}||p_n(z)|\leq \frac{1}{2(t+\epsilon_n)^n}(t+\epsilon_n)^n\leq \frac{1}{2}$. Since $\mathcal{D}$ is smooth, this means some fixed proportion $\alpha > 0$ of the ball $B\left(z_{1,n},\frac{1}{2(t+\epsilon_n)^n}\right)$ will lie in $\Lambda_{P_n}$ and this gives the required lower bound on the area. It may happen that if $K$ has bad regularity, then $\sup_{K}U_{\mu} >\log t$ and there is no way to control how much larger it will be. What follows below is to circumvent this situation.

\textbf{Proof of the claim:} We divide the proof into two steps. First, we show that, for all sufficiently large $n$, the polynomial $P_n$ has a zero converging to a \emph{regular point} of $K$ (see definition below). Second, we prove that the connected component containing this zero has sufficiently large area, thereby completing the proof. 

\textbf{{Step~1}:}  Let the polynomial $P_n$ have zeros $\{z_{n,1},\cdots ,z_{n,n}\}$, $n\in\N$. Then by the hypothesis of the claim
    \begin{align*}
        \sigma_n := \frac{1}{n} \sum_{j=1}^n\delta_{z_{j,n}} \xrightarrow[]{w^\ast}\nu_K.
    \end{align*}

By Frostman's Theorem \cite[Theorem $3.3.4$]{Ransford}, we have $U_{\nu_K}(z) = \log t$ for every $z\in K$, except possibly for a negligible Borel set of capacity $0$ (where $U_{\nu_K} > \log t)$. The bad set consists of points that are \emph{irregular} and its complement is the \emph{regular} points of $K$, i.e., \[ B:=\left\{z\in K: U_{\nu_K}(z)>\log t\right\} \qquad \& \qquad R:= K \setminus B.\]
We will show that there exist a regular point \(z_0\in\operatorname{supp}(\mu_K) \cap R\), and indices \(j_n\in\{1,\dots,n\}\) defined for all sufficiently large \(n\), such that \begin{align}\label{the point z_0}
z_{j_n,n}\longrightarrow z_0. 
\end{align}
We know that $B$ is of capacity $0$, and $\operatorname{cap}(R)=\operatorname{cap}(K)$. Furthermore, from \cite[Theorem 3.2.3]{Ransford} one has $\nu_K(B)=0$ and $\nu_K(R)=1$. Since $\operatorname{supp}(\nu_K) \cap R \neq \phi,$ we can choose a point $z_0 \in \operatorname{supp}(\nu_K) \cap R$ such that 
$\nu_K\left(B\left(z_0,1/m\right)\right)>0$ for all $m \in \N$.
    By the Portmanteau theorem \cite[Theorem $2.1$]{Bilingsleyconvergenceofprobabilitymeasures}, for each fixed $m$
    \[ \liminf_{n\to\infty}\sigma_n\left(B\left(z_0,\frac{1}{m}\right)\right) \geq \nu_K\left(B\left(z_0,\frac{1}{m}\right)\right)>0. \] 
    That is, for each $m$ there exists a $N_m$ ( w.l.o.g we assume $N_m$ to be increasing in $m$) such that 
    \begin{align}
        \sigma_n\left(B\left(z_0,\frac{1}{m}\right)\right) >0, \qquad \mbox{for}\quad  n \geq N_m.
    \end{align}
    Now for $n\in [N_m,N_{m+1}],$ choose the zero $z_{j_n,n}$ that is inside the ball  $B\left(z_0,1/m\right)$. As $m \to \infty$ it is immediate that \eqref{the point z_0} holds. Henceforth we assume that the zeros have been labelled so that $j_n=1$.
    
    \textbf{{Step}~2:} We start by defining the polynomial $p_n$ of degree $n-1$ from $P_n$ as follows. 
    \begin{align*}
        p_n(z)= \frac{P_n(z)}{z-z_{1,n}}=\prod_{j=2}^n (z-z_{n,j}).
    \end{align*}Let $\mu_n:= \frac{1}{n-1} \sum_{z: p_n(z)=0} \delta_z$ be the empirical measure of the polynomials $p_n$. It is obvious that $\mu_n\xrightarrow[]{w^\ast} \nu_K.$ 
Set $\delta_n:= \exp\left(-n(\log t)\right).$ With this $\delta_n$ we claim that there exists a $\eta_n=o(1)$, such that
\begin{align}\label{poly bound in larger set}
       \,\,  \sup_{\overline{B(z_{1,n}, \delta_n)}} \, U_{\mu_n} \leq \log t + \eta_n.
    \end{align}
Suppose \eqref{poly bound in larger set} does not hold. Then there exists  $\eta_0>0$, and a subsequence $n_k$ such that 
\begin{align}\label{for contradiction}
    U_{\mu_{n_k}}(w_k) > \log t + \eta_0,
\end{align}
where the points $w_k \in \overline{B(z_{1,n_k}, \delta_{n_k})}$. First, we show that the sequence $w_k $ converges to the point $z_0$, which was chosen in \eqref{the point z_0} of  Step~1. Indeed, by construction,  we have $z_{1,n} \to z_0$. Then an application of the triangle inequality gives the claim:
\begin{align*}
    |w_k - z_0| \leq  |w_k- z_{1,n_k}|+| z_{1,n_k}-z_0|\leq \delta_{n_k}+ | z_{1,n_k}-z_0|.
\end{align*}
Next, we employ the principle of descent (c.f. \cite[Theorem-$6.8$]{SaffTotikLogpotential}) which states that for probability measures $\sigma_n$, $n = 1, 2, ...$ having support in a fixed compact set and converging to some measure $\sigma$ in the $\mbox{weak}^{*}$ topology, we have 
    \begin{align*}
        \underset{n \to \infty }{\overline{\lim}} U_{\sigma_{n}}(y_{n}) \leq  U_{\sigma}(y_0),
    \end{align*}
    whenever $y_n \to y_0$ in the complex plane. Applying this result in the estimate \eqref{for contradiction} with $y_n = w_n$, $y_0 = z_0$, $\sigma_n=\mu_n$, and $\sigma = \nu_K$, we obtain $U_{\nu_K}(z_0)\geq \log t + \eta_0$. This is the desired contradiction since $z_0$ is a regular point. Therefore, \eqref{poly bound in larger set} holds and for all large $n \in \N$ we obtain
     \begin{align*}
     \sup_{\overline{B(z_{1,n}, \delta_n)} }\, |p_n| \ \leq \exp\left(n(\log t+\eta_n)\right).
    \end{align*}
    Now, we have all the ingredients to show that 
    \begin{align*}
B(z_{n,1},\Tilde{\delta}_n) \subset \Lambda_{P_n}, \qquad \mbox{for} \qquad \Tilde{\delta}_n:= \frac{1}{2} \exp\left(-n(\log t+\eta_n)\right) \leq \delta_n.
    \end{align*}
    Indeed, for $z \in B(z_{n,1},\Tilde{\delta}_n)$
    \begin{align}
        |P_n(z)|= |p_n(z)||z-z_{n,1}| \leq   \exp\left(n(\log t+\eta_n)\right)\, \Tilde{\delta}_n\,=\frac{1}{2}<1.
    \end{align}
    Therefore, $m(\Lambda_{P_n}) \geq \frac{\exp\left(-2n(\log t+\eta_n)\right)}{4}$. Now taking $\log$, dividing by $n$ and taking the liminf as $n \to \infty$ we obtain \eqref{lower limit}. This completes the proof of the claim and hence of the Theorem.
\end{proof}

We next tackle the question of whether Fekete polynomials are true area minimizers, not just asymptotic ones. Perhaps surprisingly, we can show the following
\begin{prop}\label{nonminimizer}
    Let $t>1.$ Then there exists a compact set $K \subset \C$ with $ \operatorname{cap}(K)=t$ for which Fekete polynomials \emph{are not} the area minimizing polynomials. 
\end{prop}
\begin{proof}
     For ease of exposition, we assume $t=2$. Let \(K= 2\overline{\D} \cup \{a\}\), where $a > 400$. Recall that $\operatorname{cap}(2\overline{\D}) = 2$.  Adding a point does not change capacity, so we have $ \operatorname{cap}(K)=2$.  Since Fekete points maximize pairwise product of distances, it is obvious that for all $n \in \N$, the Fekete polynomials $f_n$ have a simple zero at $a$. It is well known that all points of $2\overline{\D}$ are regular in the sense of Step $1$ of the proof of Theorem~\ref{asymptotic min}. Therefore for large enough $n$ we can find zeros of $f_n$ near the points $-2$ and $2$. We rename those zeros as $b_n$ and $c_n$ respectively. The other zeros will be denoted by $z_{j,n}$ for $j=\{3, 4\cdots, n\}$. Let us define a new polynomial $g_n$ as follows
    \begin{align*}
        g_n(z) = f_n(z) \, \frac{(z-a)}{(z-c_n)} =   (z-a) (z-b_n) \prod_{j=3}^n (z-z_{j,n}).
    \end{align*}
     We note that to obtain $g_n$ from $f_n$, the zero at $c_n$ has been removed and in its place, an additional zero is placed at the point $a$.  In particular, $g_n$ has all zeros inside $K$ and a zero of multiplicity $2$ at the point $a$. Since Fekete polynomials have distinct zeros, this will imply that $g_n$ is not a Fekete polynomial for $K$. We now claim that  
    \begin{align}\label{area inequality0}
        m(\Lambda_{f_n})> m\left(\Lambda_{g_n}\right),
    \end{align}
which will prove that Fekete polynomials are not true area minimizers for such a compact. Towards proving \eqref{area inequality0}, we recall a result of Kovari and Pommerenke (c.f. note added in the proof, page 75 \cite{Kovaripommerenke}), that the minimum distance between \emph{any two} Fekete points is at least of the order $n^{-2}$. On the other hand, the proof of Theorem \ref{main} shows that the diameter of each component of $\Lambda_{f_n}$ is exponentially small. This means that every component of $\Lambda_{f_n}$ contains exactly one zero (we henceforth call such components \emph{isolated}). Similarly, all components of $\Lambda_{g_n}$ except the component containing $a$ are \emph{isolated} components. 

\vspace{0.1in}

We will now analyze the effect of modifying $f_n$ to $g_n$ on the size of their respective components. Let $\Omega_w(P)$ be the connected component of $\Lambda_P$  containing $w$. By the definition of $g_n$, we have  $\frac{g_n(z)}{f_n(z)} = \frac{z-a}{z-c_n}$. By our choice of $a$, this implies that $|g_n| > |f_n|$ on every component of $\Lambda_{f_n}$, except $\Omega_a(f)$ (abusing notation here for simplicity). From here it follows that every component of $\Lambda_{g_n}$ (except $\Omega_{a}(g)$) is contained within the corresponding component of $\Lambda_{f_n}$. In contrast, since $g_n$ has a double zero at $a$, we can see that $\Omega_{a}(f)\subset\Omega_a(g)$. Indeed, a simple triangle inequality shows $B\left(a, \frac{1}{a^{3n/4}}\right)\subset\Omega_{a}(g)$. At the same time, the estimate $|f_n'(a)| \geq (a-2)^n$ is easy to see, and followed by an application of Lemma \ref{mainlemma} below (say with $r=2$) gives $\Omega_{a}(f)\subset B\left(a, \frac{4}{(a-2)^{n}}\right)\subset B\left(a, \frac{1}{a^{3n/4}}\right)$ by our choice of $a$. The situation is now clear: by modifying $f_n$ to $g_n$, the area of $\Omega_{b_n}(f)$ and the areas of all $\Omega_{z_{j,n}}(f)$ have decreased, while the area of $\Omega_a(f)$ has increased. Finally, the component $\Omega_{c_n}(f)$ has been deleted altogether. In other words, we have
    \begin{align*}
        m\left(\Lambda_{{f}_n}\setminus \big(\Omega_a(f)  \cup\Omega_{b_n}(f)\big)\right)\geq m\left(\Lambda_{g_n}\setminus \Big(\Omega_a(g) \cup \Omega_{b_n}(g)\Big)\right).
    \end{align*}
    Therefore by the preceding arguments, to show \eqref{area inequality0} it suffices to prove
    \begin{align}\label{area inequality 2}
       & m\left(\  \Omega_{b_n}(f)\setminus \Omega_{b_n}(g)\right) >  m\left(\Omega_a(g)\setminus  \Omega_a(f) \right).
    \end{align}
    Since the size of an \emph{isolated} component is controlled by the derivative at the unique zero in the component, we first bound the derivatives of $f_n$ and $g_n$ at $b_n$ in the following way.
    \begin{align*}
        |f_n'(b_n)|= 2\,|a+2| \, \prod_{j=3}^n \, |b_n-z_{j,n}| \, \leq \, 4^{n} a \qquad \& \qquad  |g_n'(b_n)| = (a+2) \Theta (|f_n'(b_n)|).
    \end{align*}
    By Lemma~\ref{mainlemma} ( we can use this lemma with $r = 2$, because of Remark~\ref{remark1}), along with  Koebe's $\frac{1}{4}$ theorem applied to $\Omega_{b_n}(g)$ we have for some absolute $c> 0$,
    \begin{align}\label{lower bound for area}
        m\left(\Omega_{b_n}(f) \setminus \Omega_{b_n}(g)\right) \geq \frac{\pi}{|f_n'(b_n)|^2} \left( \frac{1}{4}- \frac{c}{a^2} \right) \geq \frac{\pi}{8|f_n'(b_n)|^2} \geq \frac{\pi}{4^{2n+3} a^2}.
    \end{align}
    To estimate the size of $\Omega_{a}(g)$ we simply notice that for $z \in \partial B\left(a,\frac{1}{a^{n/4}}\right)$ we have 
    \begin{align*}
        |g_n(z)| \geq \frac{1}{a^{n/2}} \prod_{j=3}^n |z-z_{j,n}|\geq \frac{1}{a^{n/2}}|a-3|^{n-2} >1, 
    \end{align*}
    by our choice of $a$. Therefore, $\Omega_a(g) \subset B\left(a,\frac{1}{a^{n/4}}\right)$ and we obtain the following comparison 
    \begin{align}\label{upper bound on area for example}
          \frac{\pi}{a^{n/2}} = m\left( B\left(a,\frac{1}{a^{n/4}}\right)\right) \geq m\left(\Omega_a(g)\right) \geq m\left(\Omega_a(g)\setminus  \Omega_a(f) \right).
    \end{align}
    Now combining \eqref{lower bound for area} and \eqref{upper bound on area for example} and remembering that $a > 400$, we have
    \begin{align*}
        m\left(\Omega_{b_n}(f) \setminus \Omega_{b_n}(g)\right) \geq \frac{\pi}{4^{2n+3} a^2} > \frac{\pi}{a^{n/2}}> m\left(\Omega_a(g)\setminus  \Omega_a(f) \right).
    \end{align*}
    This proves \eqref{area inequality 2} and hence the Proposition.
\end{proof}

While the above proposition proves that Fekete polynomials are not area minimizers in general, it would be of interest to know if they are minimzers when the set $K$ is connected. We leave this as an open question.  

\section{Minimal lemniscate area when zeros are in [-2, 2]}
We would now like to shine light on an interesting aspect of the minimal area problem for sets of capacity $1$. Our two prototypical examples will be the closed disc $\overline{\D}$ and the interval $I = [-2, 2]$, both of which have capacity $1$. In the recent work \cite{KLR2}, it has been shown that $A_n(\overline{\D})\geq\frac{c}{\log n}$ for some universal constant $c>0$. On the other hand, we will prove below that $A_n(I)$ decays at a polynomial rate. This shows that while  $\vartheta(K) = 0$ may hold for all compacts $K$ with $\operatorname{cap}(K) = 1,$ the decay rate of $A_n(K)$ is not universal. It will be interesting to determine whether $A_n(K)$ can only decay like $\frac{1}{n}$ or $\frac{1}{\log n}$, or whether other rates are possible.

\begin{thm}\label{A_n(I)}
There exists a constant $c_1  > 0$ such that for all large $n$, we have
\[ A_n(I)\leq \frac{c_1}{n}.\]
\end{thm}

\begin{remark}\label{remark2}
A matching lower bound for $A_n(I)$ of the form $A_n(I)\geq\frac{c}{n}$ can be shown using methods from a forthcoming paper, but since this is not necessary to make our point about the non-universality of the decay rate of $A_n(K)$, we omit it here.     
\end{remark}

\begin{proof}[Proof of Theorem \ref{A_n(I)}]
Our proof is rather direct. Let $\mathcal{T}_n(x)$ be the monic Chebyshev polynomial of degree $n$ associated with the interval \([-2,2]\), that is
\[
\mathcal{T}_n(x):=2\,T_n(x/2),
\]
 where \(T_n\) denotes the classical Chebyshev polynomial of the first kind, defined by
\[
T_n(\cos \theta)=\cos(n\theta).
\]
 To prove Theorem \ref{A_n(I)} it suffices to show that
\begin{align}\label{bound for chebychev}
    m(\Lambda_{\mathcal{T}_n})\leq \frac{81\pi}{2n}.
\end{align} 
The main idea in proving \eqref{bound for chebychev} is to show that the components of the Chebyshev lemniscate can be covered by sufficiently small balls centered at the corresponding zeros; see Figure \ref{fig:full_width_image}. For this purpose, we need the following lemma.

\begin{Lemma}\label{mainlemma}
Let \(p\) be a monic polynomial, and let \(\Omega_1\) be a connected component of $\Lambda_p$ which contains exactly one zero $z_0$. Let \(r>1\), and let \(\Omega_r\) be the connected component of
\[
\{z\in \mathbb C: |p(z)|<r\},
\]
which contains \(\Omega_1\). Assume further that \(\Omega_r\) contains no zero of \(p\) other than \(z_0\). Then
\begin{align}\label{Ball}
    \Omega_1
\subset B\!\left(z_0,\frac{r^2}{(r-1)^2\,|p'(z_0)|}\right).
\end{align}  
\end{Lemma}
\begin{figure}
    \centering
    \includegraphics[width=\textwidth]{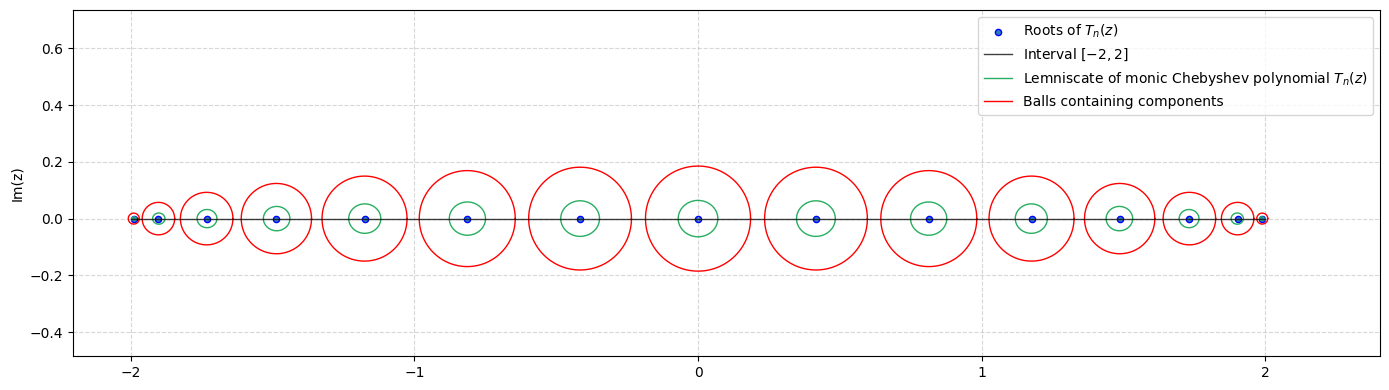}
     \caption{The lemniscate of the monic Chebyshev polynomial of degree \(15\) is shown in green. The disjoint red balls, which contain the connected components of the lemniscate, are of the form \eqref{Ball} with \(r=1.5\).}
    \label{fig:full_width_image}
\end{figure}
Assuming the Lemma for now, we will show how to conclude the proof of Theorem \ref{A_n(I)}.  We first note that all the critical values of $\mathcal{T}_n(x)$ are $2$ in modulus \cite[Example~$2.7$]{GhoshRamachandran}. This implies $\Lambda_{\mathcal{T}_n}$ has $n$ connected components (see \cite[Lemma~2.5]{GhoshRamachandran}). With this in hand, for each connected component of $\Lambda_{\mathcal{T}_n}$ we can use Lemma \ref{mainlemma} with $r=3/2$. Let \(x_1,\dots,x_n\) denote the zeros of \(\mathcal{T}_n\). Then, in view of Lemma \ref{mainlemma} the bounds in \eqref{bound for chebychev} will follow once we prove that
\begin{equation}\label{main thing to show in chebyshev}
    \sum_{k=1}^n \frac{1}{\lvert \mathcal{T}_n'(x_k)\rvert^2}=\frac{1}{2n}.
\end{equation}

The zeros of \(\mathcal{T}_n\) are
\[
x_k=2\cos\theta_k,
\qquad
\theta_k:=\frac{(2k-1)\pi}{2n},
\qquad
k=1,\dots,n.
\]
A simple computation shows that
\[
|\mathcal{T}_n'(x_k)|
=
n\,\frac{|\sin(n\theta_k)|}{|\sin\theta_k|}=n\,\frac{\Big|\sin\big(\frac{(2k-1)\pi}{2}\big)\Big|}{|\sin\theta_k|}=\frac{n}{\sin\theta_k}.
\]
Summing over \(k=1,\dots,n\), we obtain
\[
\sum_{k=1}^n \frac{1}{|\mathcal{T}_n'(x_k)|^2}
=
\frac{1}{n^2}\sum_{k=1}^n \sin^2\theta_k.
\]
Thus, it remains to evaluate the trigonometric sum above. Using $\sin^2\theta=\frac{1-\cos(2\theta)}{2}$, we get
\[
\sum_{k=1}^n \sin^2\theta_k
=
\frac12\sum_{k=1}^n \bigl(1-\cos(2\theta_k)\bigr)
=
\frac n2-\frac12\sum_{k=1}^n \cos(2\theta_k).
\]
Therefore, to finish the proof, it suffices to verify that
\begin{align}\label{cos identity}
    \sum_{k=1}^n \cos\!\left(\frac{(2k-1)\pi}{n}\right)=0,
\end{align}
This is immediate since the LHS of \eqref{cos identity} is the real part of 
\[
\sum_{k=1}^n e^{i(2k-1)\pi/n}
=
e^{i\pi/n}\sum_{k=0}^{n-1} e^{2\pi i k/n}=0.
\]
 This completes the proof of the Theorem modulo the Lemma which we now prove. 
\end{proof}

\begin{proof}[Proof of Lemma \ref{mainlemma}]
We start by noting that the hypotheses guarantee that $p$ maps $\Omega_r$ conformally onto $r\D$ with $p(\Omega_1) = \D$. Let $f:r\D\rightarrow \Omega_r$ be the inverse of $p$. For later use we note here that $f'(0) = \frac{1}{p'(z_0)}$. Now define
\[
H(\zeta):=\frac{f(r\zeta)-z_0}{r\,f'(0)},\qquad \zeta\in \mathbb D.
\]
Then \(H\) is injective holomorphic on \(\mathbb D\) with $H(0)=0$ and $H'(0)=1$.
So \(H\) is a normalized univalent function (also called Schlicht function in the classical literature) on \(\mathbb D\) and for each $z\in \Omega_1$ there exists \(w\in \mathbb D\) such that
\[
z=f(w)=f(r\zeta_w),\qquad \zeta_w:=\frac{w}{r}.
\]
From the definition of the function $H$ it follows that
\begin{align}\label{distortionineq}
|z-z_0|=|f(w)-z_0|=|f(r\zeta_w)-z_0|=r\,|f'(0)|\,|H(\zeta_w)|.
\end{align}
We now invoke the Distortion theorem for univalent functions \cite[Theorem $1.6$]{gameline} which says  that for every \(\zeta\in \mathbb D\), we have
\[
|H(\zeta)|\le \frac{|\zeta|}{(1-|\zeta|)^2}.
\]
Since \(|\zeta_w|<1/r\), feeding the estimate for $H(\zeta_w)$ into equation \eqref{distortionineq} we obtain
\[
|z-z_0|
\le
r\,|f'(0)|\cdot \frac{r}{(r-1)^2}
=
\frac{r^2}{(r-1)^2}\,\frac{1}{|p'(z_0)|}.
\]
Since \(z\in \Omega_1\) was arbitrary, this proves that
\[
\Omega_1\subset B\!\left(z_0,\frac{r^2}{(r-1)^2\,|p'(z_0)|}\right).
\]
\end{proof}
\begin{remark}
    While \eqref{bound for chebychev} provides only an upper bound for the area of 
    \(\Lambda_{\mathcal{T}_n}\), the estimate is sharp up to constants. 
    In fact, an application of Koebe's one-quarter theorem with a similar argument yields 
    the complementary lower bound
    \[
        m(\Lambda_{\mathcal{T}_n}) \geq \frac{\pi}{32n}.
    \]
\end{remark}

With the upper bound for $A_n(I)$ established, we can transplant this result to obtain bounds for $A_n(K)$ for a special class of compacts $K$ which we now define.

\begin{defn}[\textbf{Period $d$ set}, \cite{asymptoticofchebyshev,GhoshRamachandran}]
Let $K \subset \mathbb{C}$ be a compact set of logarithmic capacity $1$. Let $d\in\N$. We say that $K$ is a \emph{period-$d$ set} if there exists a monic polynomial $Q_d$ of degree $d$ such that
\begin{equation}\label{period m set}
    K = Q_d^{-1}([-2,2]).
\end{equation}
In this case, $Q_d$ is called a \emph{generating polynomial} for $K$.
\end{defn}

\begin{Lemma}\label{lem1}
    Let $K$ be a period-$d$ set. Then $\vartheta(K)=0$.
\end{Lemma}
The proof of this Lemma relies on the following result due to Crane regarding the area of polynomial preimages.

\begin{thm}\label{area theorem for lemniscate} \cite[Theorem $2$]{Crane}
    Let $p$ be a monic polynomial of degree $n$ over $\mathbb{C}$. Let $K$ be any measurable subset of the plane. Then 
\begin{equation*}
    \text{Area}(p^{-1}(K)) \leq  \pi \left(\frac{\text{Area}(K)}{\pi}\right)^{1/n},
\end{equation*}
with equality if and only if $K$ is (up to sets of measure zero) a disc and $p$ has a unique critical value at the centre of that disc.
\end{thm}
\begin{proof}[Proof of Lemma \eqref{lem1}]
Let \(Q_d\) be a generating polynomial for \(K\), and let \(\mathcal{T}_n\) denote the monic Chebyshev polynomial of \([-2,2]\) introduced above. Consider the sequence of monic polynomials
\begin{equation}\label{composition of chebyshev and generating polynoamil}
    P_{nd}:=(\mathcal{T}_n\circ Q_d)(z)=\mathcal{T}_n(Q_d(z)).
\end{equation}
Let \(x_1,\dots,x_n\) be the zeros of \(\mathcal{T}_n\). Then the zeros of \(P_{nd}\), counted with multiplicity, are precisely the points in
\[
\bigcup_{l=1}^n Q_d^{-1}(x_l).
\]
Since each \(x_l\in[-2,2]\) and $K=Q_d^{-1}([-2,2]),$ it follows that all the zeros of \(P_{nd}\) lie in \(K\). Next, by Theorem \ref{area theorem for lemniscate} ,
\begin{align}\label{area inequality}
\operatorname{Area}(\Lambda_{P_{nd}})
=
\operatorname{Area}\!\left(Q_d^{-1}\bigl(\mathcal{T}_n^{-1}(\mathbb D)\bigr)\right) 
\le
\pi \left(\frac{\operatorname{Area}(\mathcal{T}_n^{-1}(\mathbb D))}{\pi}\right)^{1/d}.
\end{align}
Applying the bound \eqref{bound for chebychev} from Theorem \ref{A_n(I)} in \eqref{area inequality}, we obtain that for all $n\in\mathbb N$
\[
A_{nd}(K)\le C\,n^{-1/d}.
\]
Passing to the limit as \(n\to\infty\), we conclude that $\vartheta(K)=0.$ 
\end{proof}

\subsection*{Acknowledgements} The authors are grateful to the anonymous referee whose feedback led us to explore the contents of Theorem~\ref{asymptotic min}. S.G. gratefully acknowledges support from the Israel Science Foundation, grant No. 3541/24.

\bibliographystyle{siam}
\bibliography{ref}

\end{document}